\documentclass[12pt,a4paper]{article}
\usepackage{amsmath,amssymb,amsbsy,amscd,amsthm}
\newtheorem{thm}{Theorem}[section]
\newtheorem{prop}[thm]{Proposition}
\newtheorem{lem}[thm]{Lemma}
\newtheorem{cor}[thm]{Corollary}

\newtheorem{problem}{Problem}

\newcommand{\ch }{\mathop{\rm char}\nolimits}
\newcommand{\ord }{\mathop{\rm ord}\nolimits}

\newcommand{\Aut }{\mathop{\rm Aut}\nolimits}

\newcommand{\G }{{\mathcal G}}
\newcommand{\id	}{{\rm id}}
\newcommand{\zs}{\{ 0\} }
\newcommand{\sm}{\setminus}

\newcommand{\p}{{\mathfrak{p}}}

\newcommand{\C}{{\bf C}}
\newcommand{\Q}{{\bf Q}}
\newcommand{\Z}{{\bf Z}}

\newcommand{\ep}{{\epsilon}}
\newcommand{\x}{{\bf x}}
\newcommand{\ba}{{\bf a}}
\newcommand{\kx}{k[{\bf x}]}
\newcommand{\Rx}{R[{\bf x}]}
\newcommand{\Sx}{S[{\bf x}]}
\newcommand{\Kx}{K[{\bf x}]}
\newcommand{\kapx}{\kappa [{\bf x}]}

\begin{document}

\title{Subgroups of polynomial automorphisms 
with diagonalizable fibers}

\author{Shigeru Kuroda}

\date{}


\footnotetext{Partly supported by the Grant-in-Aid for 
Young Scientists (B) 24740022, 
Japan Society for the Promotion of Science. }

\maketitle

\begin{abstract}
Let $R$ be an integral domain over a field $k$, 
and $G$ a subgroup of the automorphism group 
of the polynomial ring $R[x_1,\ldots ,x_n]$ 
over $R$. 
In this paper, 
we discuss when $G$ is diagonalizable 
under the assumption that $G$ is diagonalizable 
over the field of fractions of $R$. 
We are particularly interested in the case 
where $G$ is a finite abelian group. 
Kraft-Russell (2014) 
implies that every finite abelian subgroup of 
$\Aut _RR[x_1,x_2]$ is diagonalizable 
if $R$ is an affine PID over $k=\C $. 
One of the main results of this paper 
says that the same holds 
for a PID $R$ over any field $k$ containing 
enough roots of unity. 

\end{abstract}

\section{Introduction} 
\label{sect:intro}
\setcounter{equation}{0}

For each commutative ring $R$, 
we denote by $\Rx =R[x_1,\ldots ,x_n]$ 
the polynomial ring in $n$ variables over $R$, 
and by $\Aut _R\Rx $ the automorphism group 
of the $R$-algebra $\Rx $. 
We identify an endomorphism $\phi $ 
of the $R$-algebra $\Rx $ 
with the $n$-tuple $(\phi (x_1),\ldots ,\phi (x_n))$ 
of elements of $\Rx $, 
where the composition is defined by 
$\phi \circ \psi 
=(\phi (\psi (x_1)),\ldots ,\phi (\psi (x_n)))$. 
Note that, 
if $G$ is a subgroup of $\Aut _R\Rx $, 
and $S$ is a commutative $R$-algebra, 
then $G_S:=\{ \id _S\otimes \phi \mid \phi \in G\} $ 
is a subgroup of $\Aut _S\Sx $. 
When $S=\kappa (\mathfrak{p})$ 
is the residue field of 
the localization 
$R_{\mathfrak{p}}$ of $R$ 
at a prime ideal $\mathfrak{p}$ of $R$, 
we denote this group by $G_{\mathfrak{p}}$. 
If $R$ is a domain, 
$K$ denotes the field of fractions of $R$.

Throughout this paper, 
let $k$ be an arbitrary field. 
If $R$ is a $k$-algebra, 
then 
$D_n(k):=\{ \delta _{\ba }\mid \ba \in (k^*)^n\} $ 
is a subgroup of $\Aut _R\Rx $, 
where we define 
$\delta _{\ba }:=(a_1x_1,\ldots ,a_nx_n)$ 
for each $\ba =(a_1,\ldots ,a_n)\in (k^*)^n$. 
We say that a subgroup $G$ of $\Aut _R\Rx $ is 
{\it diagonalizable} 
if there exists $\psi \in \Aut _R\Rx $ 
such that $\psi ^{-1}\circ G\circ \psi $ 
is contained in $D_n(k)$.

Now, 
assume that $R$ is a $k$-domain. 
In this paper, 
we discuss the following problems.

\begin{problem}\label{prob:main}
Let $G$ be a subgroup of $\Aut _R\Rx $ 
such that $G_{(0)}$ is diagonalizable. 
Does it follow that $G$ is diagonalizable? 
\end{problem}

If we regard $\Aut _R\Rx $ as a subgroup of $\Aut _K\Kx $, 
then the assumption of Problem~\ref{prob:main} 
is equivalent to 
$\psi ^{-1}\circ G\circ \psi \subset D_n(k)$ 
for some $\psi \in \Aut _K\Kx $. 
When $n=2$, 
this condition implies that 
$G_{\mathfrak{p}}$ is diagonalizable 
for any prime ideal $\mathfrak{p}$ of $R$ 
by van der Kulk~\cite{Kulk} 
and Serre~\cite{tree} 
(cf.~Section~\ref{sect:action}). 
So 
we also consider the following problem 
for $n\geq 3$.

\begin{problem}\label{prob:main2}
Let $G$ be a subgroup of $\Aut _R\Rx $ 
such that $G_{\mathfrak{p}}$ is diagonalizable 
for all the prime ideals $\mathfrak{p}$ of $R$. 
Does it follow that $G$ is diagonalizable? 
\end{problem}

We are particularly interested in the case 
where $G$ is a finite abelian group. 
In fact, 
whether every finite abelian subgroup 
of $\Aut _{\C }\C [\x ]$ is conjugate 
to a subgroup of $D_n(\C )$ 
is a difficult problem with little progress 
for $n\geq 3$ 
(see~\cite{Ig} for the case $n=2$). 
This problem is a special case of 
Kambayashi's Linearization Problem~\cite{Kam}, 
and is open even for finite cyclic groups 
(cf.~\cite{KS}). 
In the case of finite cyclic groups, 
the problem is also included in the list of 
``eight challenging open problems in affine spaces" 
by Kraft~\cite{Kraft}. 
We mention that, 
over a field of positive characteristic, 
a counterexample to 
a similar problem is already given 
by Asanuma~\cite{Asanuma}. 
The situation is worse in 
the case of positive characteristic.

Under the assumptions in Problems~\ref{prob:main} 
and \ref{prob:main2}, 
there exists a subgroup $\G $ of $(k^*)^n$ 
for which 
$G_{(0)}$ is conjugate to 
$\{ \delta _{\ba }\mid \ba \in \G \} $ 
in $\Aut _K\Kx $. 
We write 
$\ba ^i:=a_1^{i_1}\cdots a_n^{i_n}$ 
for each $\ba =(a_1,\ldots ,a_n)\in \G $ 
and $i=(i_1,\ldots ,i_n)\in \Z ^n$, 
and define $M_{\G }$ to be the set of $i\in \Z ^n$ 
such that $\ba ^i=1$ for all $\ba \in \G $. 
Let $\gamma _1,\ldots ,\gamma _n$ 
be the images of 
the coordinate unit vectors of $\Z ^n$ 
in $\Gamma _{\G }:=\Z ^n/M_{\G }$. 
For each $i$, 
let $\Gamma _{\G }^{(i)}$ 
be the subgroup of $\Gamma _{\G }$ 
generated by $\gamma _j$ 
for $1\leq j\leq n$ with $j\neq i$.

The following theorem 
is the main result of this paper.

\begin{thm}\label{thm:main}
\noindent{\rm (i)} When $n=2$, 
Problem~$\ref{prob:main}$ has an affirmative answer 
in the following two cases:

\noindent {\rm (1)} $R$ is a PID. 

\noindent {\rm (2)} $R$ is a regular UFD, 
and $\Gamma _{\G }^{(1)}$ or 
$\Gamma _{\G }^{(2)}$ is not equal to $\Gamma _{\G }$.

\noindent{\rm (ii)} 
When $n\geq 3$, 
Problem~$\ref{prob:main2}$ has an affirmative answer 
if $R$ is a regular UFD, 
and at least $n-1$ of 
$\Gamma _{\G }^{(1)},\ldots ,\Gamma _{\G }^{(n)}$ 
are not equal to $\Gamma _{\G }$. 
\end{thm}

We emphasize that the base field $k$ is arbitrary 
in Theorem~\ref{thm:main}. 
When $R$ is an affine PID over $k=\C $, 
the case (1) of Theorem~\ref{thm:main} (i) 
(and hence Corollary~\ref{cor:abel} to follow) 
is contained in 
Kraft-Russell~\cite[Thm.\ 3.2]{KR}.

In Section~\ref{sect:action}, 
we derive the following corollary 
from the case (1) of Theorem~\ref{thm:main} (i) 
(see the discussion after Theorem~\ref{thm:PID}).

\begin{cor}\label{cor:abel}
Let $R$ be a PID over a field $k$, 
and $G$ a finite abelian subgroup of $\Aut _RR[x_1,x_2]$ 
with $d:=\max \{ \ord \phi \mid \phi \in G\} $. 
If $k$ contains a primitive $d$-th root of unity, 
then $G$ is diagonalizable. 
\end{cor}

Assume that $n=2$. 
We call $f\in \Kx $ a {\it coordinate} of $\Kx $ 
if there exists $g\in \Kx $ such that $K[f,g]=\Kx $. 
If this is the case, 
for each $\phi \in \Aut _K\Kx $ 
with $\phi (f)=f$, 
there exists $h\in K[f]$ such that 
$\phi (g)=(\det J\phi )g+h$, 
where $J\phi $ denotes the Jacobian matrix of $\phi $.

We have the following corollary to 
the case (2) of Theorem~\ref{thm:main} (i).

\begin{cor}\label{cor:over A}
Let $R$ be a regular UFD over a field $k$, 
and $\phi \in \Aut _RR[x_1,x_2]$ such that 
$\det J\phi $ belongs to $k\sm \{ 1\} $. 
If there exists a coordinate $f$ of $K[x_1,x_2]$ 
with $\phi (f)=f$, 
then $\langle \phi \rangle $ is diagonalizable. 
\end{cor}

Here, 
$\langle \phi \rangle $ denotes the subgroup of $\Aut _R\Rx $ 
generated by $\phi $. 
In fact, 
setting 
$u:=\det J\phi $ and 
$\psi :=(f,g+(u-1)^{-1}h)\in \Aut _K\Kx $, 
we have 
$\psi ^{-1}\circ \phi \circ \psi =(x_1,ux_2)$. 
Hence, 
$\psi ^{-1}\circ 
\langle \phi \rangle \circ \psi 
=\{ \delta _{\ba }\mid \ba \in \G \} $ 
holds for 
$\G :=\{ (1,u^i)\mid i\in \Z \} $. 
Thus, 
we get $\gamma _1=0$, 
and therefore 
$\Gamma _{\G }^{(2)}\neq \Gamma _{\G }$.

The structure of this paper is as follows. 
In Section~\ref{sect:action}, 
we recall the notion of algebraic actions of 
subgroups of $(k^*)^r$ on $\Rx $, 
and prove some preliminary results. 
We also derive a consequence of Theorem~\ref{thm:main}. 
In Section~\ref{sect:DC}, 
we study centrizer of subgroups of $D_n(k)$ 
in $\Aut _R\Rx $. 
Section~\ref{sect:Proof} is devoted to proving 
the case (1) of Theorem~\ref{thm:main} (i). 
In this proof, 
the main result of Section~\ref{sect:DC} 
is crucial when $k$ is not algebraically closed. 
The rest of Theorem~\ref{thm:main} 
is proved in Section~\ref{sect:Gauss} 
using a different technique.

\section{Algebraic $\G $-actions}\label{sect:action}
\setcounter{equation}{0}

Let $\G $ be a subgroup of $(k^*)^r$, 
where $r\geq 1$. 
As in Section~\ref{sect:intro}, 
we define $M_{\G }$ to be the set of $i\in \Z ^r$ 
such that $\ba ^i=1$ for all $\ba \in \G $. 
Then, 
for each $\ba \in \G $, 
the map $\Z ^r\ni i\mapsto \ba ^i\in k^*$ 
factors through $\Gamma _{\G }:=\Z ^r/M_{\G }$. 
We denote by $\ba ^{\gamma }$ 
the image of $\gamma \in \Gamma _{\G }$ 
under the induced map $\Gamma _{\G }\to k^*$. 
Let $R$ be a $k$-algebra. 
An {\it algebraic $\G $-action} on $\Rx $ 
is by definition a collection 
$V=(V_{\gamma })_{\gamma \in \Gamma _{\G }}$ 
of $R$-submodules of $\Rx $ such that 
$\Rx =\bigoplus _{\gamma \in \Gamma _{\G }}V_{\gamma }$, 
and $V_{\lambda }V_{\mu }\subset V_{\lambda +\mu }$ 
for each $\lambda ,\mu \in \Gamma _{\G }$. 
For each $\ba \in \G $, 
we define an automorphism 
$\phi _{\ba }^V:\Rx \to \Rx $ 
by $\phi _{\ba }^V(f)=\ba ^{\gamma }f$ 
for $f\in V_{\gamma }$ and $\gamma \in \Gamma _{\G }$. 
Then, the map 
$$
\rho ^V:\G \ni \ba \mapsto \phi _{\ba }^V\in \Aut _R\Rx 
$$ 
is a homomorphism of groups. 
We say that $f\in \Rx $ is $V$-{\it homogeneous} 
if $f$ belongs to $V_{\gamma }$ 
for some $\gamma \in \Gamma _{\G }$. 
Note that, 
for each $\gamma \in \Gamma _{\G }\sm \zs $, 
there exists $\ba \in \G $ 
such that $\ba ^{\gamma }\neq 1$. 
Hence, 
$f\in \Rx $ is $V$-homogeneous 
if and only if $\phi (f)\in kf$ 
for all $\phi \in \rho ^V(\G )$. 
We say that $V$ is {\it diagonalizable} 
if the subgroup $\rho ^V(\G )$ of $\Aut _R\Rx $ 
is diagonalizable, 
or equivalently there exists 
$\psi \in \Aut _R\Rx $ such that 
$\psi (x_1),\ldots ,\psi (x_n)$ are $V$-homogeneous. 
We remark that 
$V_{\gamma }\cap R=\zs $ holds 
for any $\gamma \neq 0$. 
If $S$ is an $R$-algebra, 
then 
$V_S:=(S\otimes _RV_{\gamma })_{\gamma \in \Gamma _{\G }}$ 
is an algebraic $\G $-action on $\Sx $ 
with $\rho ^{V_S}(\G )=\rho ^V(\G )_S$.

The following lemma holds for any 
algebraic $\G $-action on $\Rx $.

\begin{lem}\label{lem:kspan}
Let 
$f=\sum _{\gamma \in \Gamma _{\G }}f_{\gamma }\in \Rx $, 
where $f_{\gamma }\in V_{\gamma }$. 
Then, 
for each $\gamma \in \Gamma _{\G }$, 
we may write $f_{\gamma }$ as a $k$-linear combination 
of $\phi (f)$ for $\phi \in \rho ^V(\G )$. 
\end{lem}
\begin{proof}
We prove the lemma by induction on 
$l:=\# \{ \gamma \mid f_{\gamma }\neq 0\} $. 
The assertion is clear if $l\leq 1$. 
Assume that $l\geq 2$. 
Take 
$\lambda ,\mu \in 
\{ \gamma \mid f_{\gamma }\neq 0\}$ 
with $\lambda \neq \mu $, 
and 
$\ba \in \G $ with 
$\ba ^{\lambda }\neq \ba ^{\mu }$. 
For each $\alpha \in k^*$, 
we define $g_{\alpha }$ to be the sum of 
$f_{\gamma }$ for $\gamma \in \Gamma _{\G }$ 
with $\ba ^{\gamma }=\alpha $. 
Then, 
we have $f=\sum _{\alpha \in k^*}g_{\alpha }$ 
and $\phi (g_{\alpha })=\alpha g_{\alpha }$ 
for each $\alpha \in k^*$, 
where $\phi :=\phi _{\ba }^V$. 
Let $\alpha _1,\ldots ,\alpha _s$ 
be distinct elements of $k^*$ 
such that $f=\sum _{i=1}^sg_{\alpha _i}$. 
Then, 
by linear algebra, 
each $g_{\alpha _i}$ 
is written as a $k$-linear combination of 
$\phi ^j(f)
=\sum _{i=1}^s\alpha _i^jg_{\alpha _i}$ 
for $0\leq j<s$. 
Since $g_{\alpha }\neq 0$ 
for $\alpha =\ba ^{\lambda },\ba ^{\mu }$, 
the number of nonzero $V$-homogeneous 
components of $g_{\alpha }$ 
is less than $l$ for each $\alpha $. 
Hence, 
the lemma follows by induction assumption. 
\end{proof}

Now, 
assume that $n=2$. 
Recall the following fact 
which is a consequence of 
van der Kulk~\cite{Kulk} and Serre~\cite{tree} 
(see also \cite[Prop.\ 1.11]{Wright}): 
Let $G$ be a subgroup of $\Aut _k\kx $ 
such that 
$\deg G:=\{ \deg \phi (x_i)\mid \phi \in G,\ i=1,2\} $ 
is bounded above. 
Here, 
$\deg f$ denotes the total degree of $f$ 
for a polynomial $f$. 
Then, 
$G$ is conjugate to a subgroup of 
the {\it affine subgroup} 
$$
\mathfrak{A}_2(k):=\{ 
\phi \in \Aut _k\kx 
\mid \deg \phi (x_1)=\deg \phi (x_2)=1\} 
$$ 
or the {\it Jonqui\`ere subgroup} 
$$
\mathfrak{J}_2(k)
:=\{ (ax_1+c,bx_2+f(x_1))\mid 
a,b\in k^*,\ c\in k,\ f(x_1)\in k[x_1]\} . 
$$

The following proposition 
is a consequence of this fact.

\begin{prop}\label{prop:serre}
If $k'$ is an extension field of $k$, 
then every algebraic $\G $-action on $k'[x_1,x_2]$ 
is diagonalizable. 
\end{prop}
\begin{proof}
For $i=1,2$, 
we write 
$x_i=\sum _{\gamma \in \Gamma _{\G }}x_{i,\gamma }$, 
where $x_{i,\gamma }\in V_{\gamma }$. 
Then, 
we have $\phi _{\ba }^V(x_i)
=\sum _{\gamma \in \Gamma _{\G }}\ba ^{\gamma }x_{i,\gamma }$ 
for each $\ba \in \G $. 
Hence, 
$\deg \rho ^V(\G )$ is bounded above by 
$\max \{ \deg x_{i,\gamma }\mid 
\gamma \in \Gamma _{\G },i=1,2\} $. 
By the fact above, 
there exists $\psi =(f_1,f_2)\in \Aut _k'k'[\x ]$ 
such that $G':=\psi ^{-1}\circ \rho ^V(\G )\circ \psi $ 
is contained in $\mathfrak{A}_2(k')$ or $\mathfrak{J}_2(k')$. 
Write $f_i=\sum _{\gamma \in \Gamma _{\G }}f_{i,\gamma }$ 
for $i=1,2$, 
where $f_{i,\gamma }\in V_{\gamma }$. 
When $G'\subset \mathfrak{A}_2(k')$, 
we have $\phi (f_i)\in k'f_1+k'f_2+k'$ 
for each $\phi \in \rho ^V(\G )$ and $i=1,2$. 
Hence, 
$f_{i,\gamma }$ belongs to 
$k'f_1+k'f_2+k'$ for each $i$ and $\gamma $ 
by Lemma~\ref{lem:kspan}. 
This implies that 
$k'[f_{1,\lambda },f_{2,\mu }]=k'[f_1,f_2]$ 
for some $\lambda ,\mu \in \Gamma _{\G }$. 
Since $(f_{1,\lambda },f_{2,\mu })\in \Aut _k'k'[\x ]$, 
and $f_{1,\lambda }$ and $f_{2,\mu }$ 
are $V$-homogeneous, 
we conclude that $V$ is diagonalizable. 
When $G'\subset \mathfrak{J}_2(k')$, 
we have $f_{1,\gamma }\in k'f_1+k'$ 
and $f_{2,\gamma }\in k'f_2+k'[f_1]$ 
for each $\gamma \in \Gamma _{\G }$ 
by Lemma~\ref{lem:kspan}. 
From this, 
the assertion follows similarly. 
\end{proof}

The following theorem is a consequence of 
the case (1) of Theorem~\ref{thm:main} (i), 
since 
$\rho ^V(\G )_{(0)}=\rho ^{V_K}(\G )$ 
is diagonalizable by Proposition~\ref{prop:serre}.

\begin{thm}\label{thm:PID}
Let $R$ be a PID over a field $k$, 
and $\G $ a subgroup of $(k^*)^r$ for some $r\geq 1$. 
Then, 
every algebraic $\G $-action on $R[x_1,x_2]$ 
is diagonalizable. 
\end{thm}

When $R$ is an affine PID over $k=\C $, 
Theorem~\ref{thm:PID} 
is contained in 
Kraft-Russell~\cite[Thm.\ 3.2]{KR}. 
In fact, 
they treated actions of 
reductive groups more generally.

Corollary~\ref{cor:abel} 
is derived from Theorem~\ref{thm:PID} as follows. 
Let $\phi _1,\ldots ,\phi _r\in G$ 
be such that 
$G=\langle \phi _1\rangle 
\times \cdots \times \langle \phi _r\rangle $. 
Then, 
since 
$d_i:=\ord \phi _i$ is a divisor of $d$, 
there exists a primitive $d_i$-th root 
$\zeta _i\in k$ of unity for $i=1,\ldots ,r$. 
Set $\G =
\{ (\zeta _1^{i_1},\ldots ,\zeta _r^{i_r})\mid 
i_1,\ldots ,i_r\in \Z \} $. 
Then, 
$\Gamma _{\G }$ is equal to 
$\prod _{i=1}^r(\Z /d_i\Z )$. 
For each 
$\gamma =(\bar{i}_1,\ldots ,\bar{i}_r)\in \Gamma _{\G }$, 
we define 
$V_{\gamma }$ to be the set of $f\in \Rx $ 
for which $\phi _l(f)=\zeta _l^{i_l}f$ 
holds for $l=1,\ldots ,r$. 
Then, 
for each $f\in \Rx $, 
we have 
$$
f_{\gamma }:=|G|^{-1}
\sum _{l_1=0}^{d_1-1}\cdots \sum _{l_r=0}^{d_r-1}
\zeta _1^{-i_1l_1}\cdots \zeta _r^{-i_rl_r}
(\phi _1^{l_1}\circ \cdots \circ \phi _r^{l_r})(f)
\in V_{\gamma }. 
$$
Since 
$f=\sum _{\gamma \in \Gamma _{\G }}f_{\gamma }$, 
we see that 
$V=(V_{\gamma })_{\gamma \in \Gamma _{\G }}$ 
is an algebraic $\G $-action on $\Rx $ with 
$\rho ^V(\G )=G$. 
Hence, 
$G$ is diagonalizable by Theorem~\ref{thm:PID}.

Next, 
let $\G $ be a subgroup of $(k^*)^n$. 
For each $\gamma \in \Gamma _{\G }$, 
we define $\Rx _{\gamma }$ to be the $R$-submodule 
of $\Rx $ generated by the monomials 
$x_1^{i_1}\cdots x_n^{i_n}$ 
such that the image of 
$(i_1,\ldots ,i_n)$ in $\Gamma _{\G }$ 
is equal to $\gamma $. 
Then, 
we have 
\begin{equation}\label{eq:standard grading}
\Rx _{\gamma }
=\{ f\in \Rx \mid \delta _{\ba }(f)=\ba ^{\gamma }f
\text{ for all }\ba \in \G \} . 
\end{equation}
Hence, 
$V=(\Rx _{\gamma })_{\gamma \in \Gamma _{\G }}$ 
is an algebraic $\G $-action on $\Rx  $ 
such that $\phi _{\ba }^V=\delta _{\ba }$ 
for each $\ba \in \G $.

Now, 
assume that $R$ is a $k$-domain, 
and let $G$ be a subgroup of $\Aut _R\Rx $ 
such that $G_{(0)}$ is diagonalizable. 
Then, 
there exists $\psi \in \Aut _K\Kx $ 
and a subgroup $\G $ of $(k^*)^n$ such that 
$\psi ^{-1}\circ G\circ \psi 
=\{ \delta _{\ba }\mid \ba \in \G \} $. 
When this is the case, 
$V_K:=(\psi (\Kx _{\gamma }))_{\gamma \in \Gamma _{\G }}$ 
is an algebraic $\G $-action on $\Kx $ 
satisfying $\rho ^{V_K}(\G )=G_{(0)}$. 
Set 
\begin{equation}\label{eq:intersection}
V_{\gamma }:=\psi (\Kx _{\gamma })\cap \Rx 
\quad\text{for each}\quad \gamma \in \Gamma _{\G }. 
\end{equation}
We claim that 
$V:=(V_{\gamma })_{\gamma \in \Gamma _{\G }}$ 
is an algebraic $\G $-action on $\Rx $ 
with $\rho ^V(\G )=G$. 
In fact, 
take any $f\in \Rx $ 
and write $f=\sum _{\gamma }f_{\gamma }$, 
where $f_{\gamma }\in \psi (\Kx _{\gamma })$. 
Then, 
by Lemma~\ref{lem:kspan}, 
each $f_{\gamma }$ is a $k$-linear combination 
of $\phi (f)\in \Rx $ 
for $\phi \in \rho ^{V_K}(\G )=G_{(0)}$. 
Thus, 
$f_{\gamma }$ belongs to $\Rx $, 
and hence to $V_{\gamma }$. 
We remark that 
$f\in \Rx $ is $V$-homogeneous 
if and only if $\phi (f)\in kf$ 
holds for each $\phi \in G$. 
In this sense, 
$V$ is uniquely defined from $G$. 
If $n=2$, 
then $V_{\kappa (\p )}$ is diagonalizable 
for any prime ideal $\mathfrak{p}$ of $R$ 
by Proposition~\ref{prop:serre}. 
Hence, 
$G_{\mathfrak{p}}$ is diagonalizable 
as remarked after Problem~\ref{prob:main}.

Finally, 
we prove a lemma 
used in Section~\ref{sect:Gauss}. 
We call a sequence 
$f_1,\ldots ,f_r$ of elements of $\Rx $ 
a {\it partial system of coordinates} of $\Rx $ 
if there exist $f_{r+1},\ldots ,f_n\in \Rx $ 
such that $R[f_1,\ldots ,f_n]=\Rx $, 
or equivalently there exists $\phi \in \Aut _R\Rx $ 
for which $\phi (x_i)=f_i$ holds for $i=1,\ldots ,r$ 
(cf.~\cite[Prop.\ 1.1.6]{Essen}). 
When $r=1$, 
such an $f_1$ is called a {\it coordinate} of $\Rx $.

\begin{lem}\label{lem:partial system}
Let $R$ be a $k$-domain, 
and $G$ a subgroup of $\Aut _R\Rx $. 
Assume that there exists 
$\psi =(f_1,\ldots ,f_n)
\in \Aut _K\Kx $ 
with 
$\psi ^{-1}\circ G\circ \psi \subset D_n(k)$ 
for which $f_1,\ldots f_{n-1}$ 
form a partial system of 
coordinates of $\Rx $. 
Then, $G$ is diagonalizable. 
\end{lem}
\begin{proof}
In the situation of the lemma, 
we may define a subgroup $\G $ of $(k^*)^n$ 
and an algebraic $\G $-action $V$ on $\Rx $ 
as above. 
Then, 
$f_1,\ldots ,f_{n-1}$ are $V$-homogeneous. 
Set $A=R[f_1,\ldots ,f_{n-1}]$ 
and $B=K[f_1,\ldots ,f_{n-1}]$. 
Then, 
we have $B\cap \Rx =A$, 
since $f_1,\ldots ,f_{n-1}$ 
is a partial system of coordinates of $\Rx $. 
There exists $g\in \Rx $ such that 
$R[f_1,\ldots ,f_{n-1},g]=\Rx $. 
It suffices to show that 
$g$ is chosen to be $V$-homogeneous. 
Write 
$g=\sum _{\gamma \in \Gamma _{\G }}g_{\gamma }$, 
where $g_{\gamma }\in V_{\gamma }$ 
for each $\gamma \in \Gamma _{\G }$. 
Since $g_{\gamma }$ belongs to $\Rx $, 
we have $g_{\gamma }\in A=B\cap \Rx $ 
if and only if $g_{\gamma }\in B$. 
We show that this holds for each $\gamma \neq \mu $, 
where $\mu \in \Gamma _{\G }$ 
is such that $f_n\in \psi (\Kx _{\mu })$. 
Then, 
it follows that 
$R[f_1,\ldots ,f_{n-1},g_{\mu }]=\Rx $, 
and the proof is complete. 
Observe that $B[g]=\Kx =B[f_n]$. 
This implies that $g=uf_n+h$ 
for some $u\in K^*$ and $h\in B$. 
Write $h=\sum _{\gamma \in \Gamma _{\G }}h_{\gamma }$, 
where $h_{\gamma }\in \psi (\Kx _{\gamma })$. 
Then, 
we have $g_{\mu }=uf_n+h_{\mu }$, 
and $g_{\gamma }=h_{\gamma }$ 
for each $\gamma \neq \mu $. 
Since 
$f_1,\ldots ,f_{n-1}$ 
are $V$-homogeneous, 
$h_{\gamma }$ belongs to $B$ 
for each $\gamma \in \Gamma _{\G }$. 
Therefore, 
$g_{\gamma }$ belongs to $B$ 
if $\gamma \neq \mu $. 
\end{proof}

\section{Centrizer}\label{sect:DC}
\setcounter{equation}{0}

Throughout this section, 
we assume that $\G $ 
is a subgroup of $(k^*)^n$ 
not equal to $\{ e\} $. 
For each $\gamma \in \Gamma _{\G }$, 
we define $\Rx _{\gamma }$ as in Section~\ref{sect:action}, 
where $R$ may be any commutative ring for the moment. 
We say that $f\in \Rx $ 
is $\G $-{\it homogeneous} 
if $f$ belongs to $\Rx _{\gamma }$ 
for some $\gamma \in \Gamma _{\G }$. 
Let $\phi =(f_1,\ldots ,f_n)$ 
be an element of $\Aut _R\Rx $. 
We say that $\phi $ is $\G $-{\it homogeneous} 
if $f_i$ belongs to $\Rx _{\gamma _i}$ 
for $i=1,\ldots ,n$. 
We remark that, 
if $R$ is a domain, 
and $f_1,\ldots ,f_n$ are $\G $-homogeneous, 
then $(f_{\sigma (1)},\ldots ,f_{\sigma (n)})$ 
is $\G $-homogeneous for some permutation 
$\sigma \in S_n$.  
Actually, 
since $\det J\phi $ belongs to $\Rx ^*=R^*$, 
the linear parts of $f_1,\ldots ,f_n$ 
are linearly independent over $R$. 
Hence, 
there exists $\sigma \in S_n$ such that 
the linear monomials $x_1,\ldots ,x_n$ 
appear in $f_{\sigma (1)},\ldots ,f_{\sigma (n)}$, 
respectively.

Assume that $R$ is a $k$-algebra. 
Then, 
in view of (\ref{eq:standard grading}), 
we see that $\phi$ 
is $\G $-homogeneous if and only if 
$\delta _{\ba }(f_i)=a_if_i$ 
for all $\ba =(a_1,\ldots ,a_n)\in \G $ 
and $i=1,\ldots ,n$, 
and hence if and only if 
$\delta _{\ba }\circ \phi =\phi \circ \delta _{\ba }$ 
for all $\ba \in \G $. 
Thus, 
the set $C_{\G }(R)$ 
of $\G $-homogeneous elements of $\Aut _R\Rx $ 
is equal to the centrizer of 
$\{ \delta _{\ba }\mid \ba \in \G \} $ 
in $\Aut _R\Rx $. 
We note that $C_{\G }(R)$ is a subgroup of 
$\Aut _R\Rx $ even if $R$ does not contain $k$.

Now, 
assume that $n=2$, 
and let $h$ be a coordinate of $\Rx $. 
Then, 
there exist 
$\phi \in \Aut _R\Rx $ and $i\in \{ 1,2\} $ 
such that $\phi (x_i)=h$. 
If furthermore $h$ is $\G $-homogeneous, 
then we have the following lemma.

\begin{lem}\label{lem:homoge autom}
Assume that $R$ is a $k$-domain 
such that $R^*\cup \zs $ is a field. 
If $h$ is a $\G $-homogeneous coordinate of $\Rx $, 
then there exist $\phi \in C_{\G }(R)$ 
and $i\in \{ 1,2\} $ 
such that $\phi (x_i)=h$. 
\end{lem}
\begin{proof}
There exists $g\in \Rx $ such that $R[h,g]=\Rx $. 
We show that 
$g$ is chosen to be $\G $-homogeneous. 
Then, 
$(g,h)$ or $(h,g)$ belongs to $C_{\G }(R)$ 
as remarked. 
Write 
$g=\sum _{\gamma \in \Gamma _{\G }}g_{\gamma }$, 
where $g_{\gamma }\in \Rx _{\gamma }$ 
for each $\gamma $. 
Then, 
we have 
\begin{equation}\label{eq:jacob indep}
\sum _{\gamma \in \Gamma _{\G }}\det J(h,g_{\gamma })
=\det J(h,g)\in R^*. 
\end{equation}
Put $\gamma _0:=\gamma _1+\gamma _2-\mu $, 
where $\mu \in \Gamma _{\G }$ is such that 
$h$ belongs to $\Rx _{\mu }$. 
Then, 
$\det J(h,g_{\gamma })$ belongs to 
$\Rx _{\gamma -\gamma _0}$ 
for each $\gamma $. 
From (\ref{eq:jacob indep}), 
it follows that 
$\det J(h,g_{\gamma _0})\neq 0$. 
This implies that $h$ and $g_{\gamma _0}$ 
are algebraically independent over $R$ 
(cf.~\cite[Prop.\ 1.1.31]{Essen}). 
We show that $g_{\gamma }$ belongs to $R[h]$ 
for each $\gamma \neq \gamma _0$. 
Then, 
it follows that 
$R[h,g_{\gamma _0}]=R[h,g]=\Rx $, 
and the proof is complete. 
Fix any $\lambda \neq \gamma _0$, 
and take $\ba \in \G $ such that 
$\ba ^{\lambda }\neq \ba ^{\gamma _0}$. 
Then, 
we have 
$$
R[h,g]=\delta _{\ba }(R[h,g])
=R[\delta _{\ba }(h),\delta _{\ba }(g)]
=R[\ba ^{\mu }h,\delta _{\ba }(g)]
=R[h,\delta _{\ba }(g)]. 
$$
Hence, 
we may write $\delta _{\ba }(g)=ug+g'$, 
where $u\in R^*$ 
and $g'\in R[h]$. 
Write 
$g'=\sum _{\gamma \in \Gamma _{\G }}g_{\gamma }'$, 
where $g_{\gamma }'\in \Rx _{\gamma }$. 
Then, 
$g_{\gamma }'$ belongs to $R[h]$ 
for each $\gamma $, 
since $h$ is $\G $-homogeneous by assumption. 
On the other hand, 
from the equality 
$$
\sum _{\gamma \in \Gamma _{\G }}\ba^{\gamma }g_{\gamma }
=\delta _{\ba }(g)=ug+g'
=\sum _{\gamma \in \Gamma _{\G }}
(ug_{\gamma }+g_{\gamma }'), 
$$
we see that 
$(\ba ^{\gamma }-u)g_{\gamma }=g_{\gamma }'$ 
holds for each $\gamma $. 
Since $g_{\gamma _0}$ and $h$ 
are algebraically independent over $R$, 
this implies that $\ba ^{\gamma _0}=u$. 
By the choice of $\ba $, 
it follows that 
$\ba ^{\lambda }\neq u$. 
Since $\ba ^{\lambda }$ and $u$ 
are units of $R$, 
and $R^*\cup \zs $ is a field by assumption, 
we know that $\ba ^{\lambda }-u$ is a unit of $R$. 
Therefore, 
$g_{\lambda }=(\ba ^{\lambda }-u)^{-1}g_{\lambda }'$ 
belongs to $R[h]$. 
\end{proof}

Next, 
let $R$ be any $k$-algebra, 
and $\mathfrak{p}$ a maximal ideal of $R$ 
with $\kappa :=R/\p $. 
For each $f\in \Rx $, 
we denote by $\bar{f}$ 
the image of $f$ in $\kapx $. 
The rest of this section 
is devoted to the proof of the following proposition. 
This result plays an important role 
in proving the case (1) of 
Theorem~\ref{thm:PID} (i).

\begin{prop}\label{thm:homo autom}
Assume that $n=2$, 
and let $h$ be a coordinate of $\kapx $. 
If $h$ is $\G $-homogeneous, 
then there exist 
$(g_1,g_2)\in C_{\G }(R)$, 
$a\in \kappa ^*$ and $i\in \{ 1,2\} $ 
such that 
$\det J(g_1,g_2)=1$ and $h=a\bar{g}_i$. 
\end{prop}

When $k$ is an algebraically closed field, 
Proposition~\ref{thm:homo autom} 
easily follows from 
Lemma~\ref{lem:homoge autom}, 
since $\kappa =k$ is contained in $R$. 
In the general case, 
Proposition~\ref{thm:homo autom} 
is proved by using 
a lifting technique of automorphisms. 
If $\psi =(f_1,\ldots ,f_n)$ 
is an endomorphism 
of the $R$-algebra $\Rx $, 
then 
$\bar{\psi }:=\id _{\kappa }\otimes \psi 
=(\bar{f}_1,\ldots ,\bar{f}_n)$ 
is an endomorphism of the $\kappa $-algebra $\kapx $. 
We remark that 
$\psi \in \Aut _R\Rx $ implies 
$\bar{\psi }\in \Aut _{\kappa }\kapx $, 
since 
$(\psi _1\circ \psi _2)^-=\bar{\psi }_1\circ \bar{\psi }_2$ 
holds for endomorphisms $\psi _1$ and $\psi _2$ 
of $\Rx $. 
We call $\psi \in \Aut _R\Rx $ a 
{\it lift} of $\phi \in \Aut _{\kappa }\kapx $ 
if $\bar{\psi }=\phi $. 
In general, 
it is not clear whether every element of 
$C_{\G }(\kappa )$ has a lift in $C_{\G }(R)$. 
We say that $\sigma \in \Aut _{\kappa }\kapx $ 
is {\it elementary} if 
there exist 
$1\leq i\leq n$ 
and $f\in \kappa [\{ x_j\mid j\neq i\} ]$ 
such that 
\begin{equation}\label{eq:elementary}
\sigma =(x_1,\ldots ,x_{i-1},x_i+f,x_{i+1},\ldots ,x_n). 
\end{equation}
Note that (\ref{eq:elementary}) 
is $\G $-homogeneous 
if and only if $f$ belongs to 
$\kappa [\x ]_{\gamma _i}$. 
If this is the case, 
there exists 
$g\in \Rx _{\gamma _i}\cap R[\{ x_j\mid j\neq i\} ]$ 
such that $\bar{g}=f$. 
Then, 
the elementary automorphism 
$$
\ep =
(x_1,\ldots ,x_{i-1},x_i+g,x_{i+1},\ldots ,x_n)
\in C_{\G }(R)
$$
is a lift of $\sigma $. 
Clearly, 
we have $\det J\ep =1$.

Now, 
let $k$ be any field. 
For each $\phi =(f_1,f_2)\in \Aut _kk[x_1,x_2]$, 
we have $\deg \phi :=\deg f_1+\deg f_2\geq 2$. 
It is well known 
(cf.\ e.g.~\cite[Thm.\ 8.5]{Cohn}) 
that, 
if $\deg \phi >2$, 
then 
there exist $c\in k^*$ and $l\geq 1$ 
such that 
\begin{equation}\label{eq:er}
\deg (f_1-cf_2^l)<\deg f_1
\quad\text{or}\quad 
\deg (f_2-cf_1^l)<\deg f_2. 
\end{equation}
Using this fact, 
we can prove the following lemma.

\begin{lem}\label{lem:her}
If $n=2$, 
then each $\phi \in C_{\G }(k)$ is written as 
$\phi =\sigma _1\circ \cdots \circ \sigma _r\circ \tau $ 
for some $r\geq 0$, 
where $\sigma _1,\ldots ,\sigma _r\in C_{\G }(k)$ 
are elementary, 
and $\tau \in D_2(k)$. 
\end{lem}
\begin{proof}
If $\sigma \in C_{\G }(k)$ is elementary 
and $\tau \in D_2(k)$, 
then 
$\tau ^{-1}\circ \sigma \circ \tau \in C_{\G }(k)$ 
is elementary. 
Hence, 
it suffices to show that 
$\phi 
=\tau \circ \sigma _1\circ \cdots \circ \sigma _r$ 
for some $\tau $ and $\sigma _1,\ldots ,\sigma _r$ 
as in the lemma. 
We prove this statement 
by induction on $\deg \phi $.

By assumption, 
$f_i:=\phi (x_i)$ belongs to 
$\kx _{\gamma _i}$ for $i=1,2$. 
First, 
assume that $\deg \phi =2$, i.e., 
$\deg f_1=\deg f_2=1$. 
If $\gamma _1\neq \gamma _2$ and $\gamma _1\neq 0$, 
then we have $f_1=ax_1$ and $f_2=bx_2+c$ 
for some $a,b\in k^*$ and $c\in k$, 
where $c\neq 0$ only if $\gamma _2=0$. 
Since $\phi =(ax_1,bx_2)\circ (x_1,x_2+c)$ 
and $(x_1,x_2+c)\in C_{\G }(k)$, 
the assertion is true. 
The case $\gamma _1\neq \gamma _2$ and $\gamma _2\neq 0$ 
is similar. 
If $\gamma _1=\gamma _2$, 
then $\gamma _1$ and $\gamma _2$ are nonzero, 
for otherwise $\Gamma _{\G }=\zs $, 
contradicting $\G \neq \{ e\} $. 
Hence, 
$f_1$ and $f_2$ have no constant terms. 
Thus, 
$\phi $ is a linear automorphism. 
In this case, 
the assertion follows from linear algebra.

Next, 
assume that $\deg \phi >2$. 
Then, 
there exist $c\in k^*$ and $l\geq 1$ 
for which one of the inequalities in (\ref{eq:er}) holds. 
Since both cases are similar, 
we assume the former case. 
In this case, 
a common monomial appears in $f_1$ and $f_2^l$. 
Since $f_i$ belongs to $\kx _{\gamma _i}$ for $i=1,2$, 
it follows that $\gamma _1=l\gamma _2$. 
Hence, 
$\sigma :=(x_1-cx_2^l,x_2)$ belongs to $C_{\G }(k)$. 
Since $\phi $ belongs to $C_{\G }(k)$ by assumption, 
$\phi \circ \sigma =(f_1-cf_2^l,f_2)$ 
also belongs to $C_{\G }(k)$. 
By (\ref{eq:er}), 
$\deg \phi \circ \sigma $ 
is less than $\deg f_1+\deg f_2=\deg \phi $. 
Therefore, 
by induction assumption, 
we may write 
$\phi \circ \sigma 
=\tau \circ \sigma _1\circ \cdots \circ \sigma _r$, 
where 
$\tau $ and $\sigma _1,\ldots ,\sigma _r$ 
are as in the lemma. 
Since 
$\phi 
=\tau \circ \sigma _1\circ \cdots \circ 
\sigma _r\circ \sigma ^{-1}$ 
with $\sigma ^{-1}=(x_1+cx_2^l,x_2)\in C_{\G }(k)$, 
the assertion holds true. 
\end{proof}

Let us complete the proof of 
Proposition~\ref{thm:homo autom}. 
Since $\kappa $ is an extension field of $k$, 
and since $h$ is a $\G $-homogeneous coordinate of 
$\kappa [\x ]$, 
there exist $\phi \in C_{\G }(\kappa )$ 
and $i\in \{ 1,2\} $ such that $\phi (x_i)=h$ 
by Lemma~\ref{lem:homoge autom}. 
By Lemma~\ref{lem:her}, 
we may write 
$\phi =\sigma _1\circ \cdots \circ \sigma _r\circ \tau $ 
for some $r\geq 0$, 
where $\sigma _1,\ldots ,\sigma _r\in C_{\G }(\kappa )$ 
are elementary, 
and $\tau =(a_1x_1,a_2x_2)$ 
with $a_1,a_2\in \kappa ^*$. 
For $j=1,\ldots ,r$, 
there exists a lift 
$\ep _j\in C_{\G }(R)$ 
of $\sigma _j$ with $\det J\ep _j=1$ as mentioned. 
Then, 
$(g_1,g_2):=\ep_1\circ \cdots \circ \ep _r$ 
belongs to $C_{\G }(R)$, 
and satisfies $\det J(g_1,g_2)=1$. 
Moreover, 
we have 
$$
\phi 
=\bar{\ep }_1\circ \cdots \circ \bar{\ep }_r
\circ \tau 
=(\ep _1\circ \cdots \circ \ep _r)^-\circ \tau 
=(\bar{g}_1,\bar{g}_2)\circ \tau 
=(a_1\bar{g}_1,a_2\bar{g}_2). 
$$
Therefore, 
we get $h=\phi (x_i)=a_i\bar{g}_i$. 
This completes the proof of 
Proposition~\ref{thm:homo autom}.

\section{Case (1) of Theorem~\ref{thm:PID} (i)}
\label{sect:Proof}
\setcounter{equation}{0}

The goal of this section is to 
prove the case (1) of Theorem~\ref{thm:PID} (i). 
We may assume that $G\neq \{ \id \} $. 
By assumption, 
there exist 
$\psi =(f_1,f_2)\in \Aut _K\Kx $ and a subgroup 
$\G $ of $(k^*)^2$ 
with $\G \neq \{ e \} $ 
such that 
$\psi ^{-1}\circ G\circ \psi 
=\{ \delta _{\ba }\mid \ba \in \G \} $. 
If $\psi $ belongs to $\Aut _R\Rx $, 
then we are done. 
Note that 
$$
(\psi \circ \sigma )^{-1}\circ G
\circ (\psi \circ \sigma )
=\{ \delta _{\ba }\mid \ba \in \G \} 
$$ 
holds for each $\sigma \in C_{\G }(K)$. 
Our strategy is to find $\sigma \in C_{\G }(K)$ 
such that $\psi \circ \sigma $ 
belongs 
to $\Aut _R\Rx $. 
There exist $a_1,a_2\in K^*$ for which 
$a_1f_1$ and $a_2f_2$ belong to $\Rx $. 
Since $\sigma _0:=(a_1x_1,a_2x_2)$ 
belongs to $C_{\G }(K)$, 
by replacing $\psi $ 
with $\psi \circ \sigma _0=(a_1f_1,a_2f_2)$, 
we may assume that $f_1$ and $f_2$ belong to $\Rx $. 
Then, 
$\det J\psi $ belongs to $\Rx \cap K^*=R\sm \zs $. 
Since $R$ is a PID, 
we may write 
$\det J\psi =\alpha p_1\cdots p_m$, 
where $\alpha \in R^*$ and $p_1,\ldots ,p_m$ 
are prime elements of $R$. 
Choose $\psi $ so that $m$ is minimal. 
Then, 
$f_1$ and $f_2$ do not belong to $p\Rx $ 
for any prime element $p$ of $R$. 
If $m=0$, 
then the proof is completed. 
Indeed, 
Keller's theorem says that, 
if $\psi \in \Aut _K\Kx $ satisfies 
$\psi (\Rx )\subset \Rx $ and $\det J\psi \in R^*$, 
then $\psi $ restricts to an element of $\Aut _R\Rx $ 
(cf.~\cite[Cor.~1.1.35]{Essen}). 
We show that $m=0$ by contradiction. 
Suppose that $m\geq 1$, 
and put $p:=p_m$. 
For $i=1,2$, 
set $g_i':=\psi ^{-1}(x_i)\in \Kx $, 
and take $b_i\in K^*$ 
so that $g_i:=b_ig_i'$ 
belongs to $\Rx \sm p\Rx $. 
Then, 
$b_ix_i=b_i\psi (g_i')=\psi (g_i)$ 
belongs to $\Rx $, 
since $\psi (\Rx )\subset \Rx $. 
Hence, 
$b_i$ is an element of $R\sm \zs $. 
Since $g_1$ and $g_2$ belong to $\Rx $, 
we have 
$b_1b_2\det J\psi ^{-1}=\det J(g_1,g_2)\in \Rx $. 
Hence, 
$b_1b_2=\det J(g_1,g_2)\cdot \det J\psi $ 
is divisible by $\det J\psi $, 
and thus by $p$. 
Therefore, 
$b_1$ or $b_2$ belongs to $pR$. 
In the following, 
we assume that $b_1$ belongs to $pR$.

Since $R$ is a PID over $k$, 
we see that $\kappa :=R/pR$ 
is an extension field of $k$. 
Consider the endomorphism 
$\bar{\psi }=(\bar{f}_1,\bar{f}_2)$ 
of the $\kappa $-algebra $\kappa [\x ]$. 
Since $f_1$ and $f_2$ do not belong to $p\Rx $ 
by assumption, 
$\bar{f}_1$ and $\bar{f}_2$ are nonzero. 
Since $\G \neq \{ e\} $, 
there exist 
$\phi \in G$, 
$i\in \{ 1,2\} $ 
and $\alpha \in k^*\sm \{ 1\} $ 
such that $\phi (f_i)=\alpha f_i$. 
Then, 
we have 
$\bar {\phi }(\bar{f}_i)
=(\phi (f_i))^-
=(\alpha f_i)^-
=\alpha \bar{f}_i\neq \bar{f}_i$. 
Hence, 
$\bar{f}_i$ does not belong to $\kappa $. 
This implies that the prime ideal 
$\ker \bar{\psi }$ of $\kappa [\x ]$ is not maximal, 
and hence is of height at most one. 
On the other hand, 
$\bar{g}_1$ is nonzero by definition, 
and satisfies $\bar{\psi }(\bar{g}_1)
=(\psi (g_1))^-=(b_1x_1)^-=0$. 
Hence, 
the height of $\ker \bar{\psi }$ is equal to one. 
Therefore, 
there exists a prime element $q$ of $\kappa [\x ]$ 
such that $\ker \bar{\psi }=q\kappa [\x ]$ 
(cf.~\cite[Thm.\ 20.1]{Matsumura}).

Since $g_1$ is an element of $\Rx $ 
which is a coordinate of $\Kx $, 
we may find a coordinate $h$ of $\kappa [\x ]$ 
such that $\bar{g}_1$ belongs to $\kappa [h]$ 
by the following lemma. 
This lemma is a consequence of 
Sathaye~\cite[Thm.\ 3]{Sathaye}.

\begin{lem}\label{lem:rat coord}
Assume that $n=2$, 
and let $R$, $p$ and $\kappa $ be as above. 
If $f\in \Rx $ is a coordinate of $\Kx $, 
then 
there exists a coordinate $h$ of $\kapx $ 
such that $\bar{f}$ belongs to $\kappa [h]$. 
\end{lem}
\begin{proof}
By assumption, 
there exists $g\in \Rx $ such that $K[f,g]=K[\x ]$. 
Since $R':=R_{(p)}$ 
is a rank-one discrete valuation ring 
with residue field $\kappa $, 
we obtain the lemma using 
the following 
result of Sathaye~\cite[Thm.\ 3]{Sathaye} 
for $A:=R'[\x ]$ and $(u,v):=(f,g)$: 
Let $R'$ be a rank-one discrete valuation ring 
with residue field $\kappa $ 
and field of fractions $K$, 
and let $A$ be an affine domain over $R'$ 
such that 
$A_0:=A\otimes _{R'}K$ 
and $A_1:=A\otimes _{R'}\kappa $ 
are polynomial rings in two variables over 
$K$ and $\kappa $, 
respectively. 
Take $u,v\in A$ such that $A_0=K[u,v]$, 
and let $B$ be the $\kappa $-subalgebra of $A_1$ 
generated by the images of $u$ and $v$ in $A_1$. 
If $B$ has 
transcendence degree two over $\kappa $, 
then we have $A_1=B$. 
If $B$ has 
transcendence degree one over $\kappa $, 
then there exist $x,y\in A_1$ such that 
$A_1=\kappa [x,y]$ 
and $B\subset \kappa [x]$. 
\end{proof}

We remark that 
$R'$ is not assumed to be of 
``equicharacteristic zero" 
in the above result of Sathaye, 
unlike his famous theorem \cite[Thm.~1]{Sathaye}. 
When $R$ contains $\Q $, 
Lemma~\ref{lem:rat coord} 
is also proved by using Rentschler~\cite{Rentschler} 
instead of Sathaye~\cite[Thm.\ 3]{Sathaye} 
(see the proof at the end of this section).

Now, 
write $\bar{g}_1=\Phi (h)$, 
where $\Phi (x_1)\in \kappa [x_1]\sm \zs $. 
Then, 
we have $\Phi (\bar{\psi }(h))=\bar{\psi }(\bar{g}_1)=0$. 
Hence, 
$\bar{\psi }(h)$ is algebraic over $\kappa $. 
Since $\bar{\psi }(h)$ is an element of $\kappa [\x ]$, 
it follows that $\bar{\psi }(h)$ belongs to $\kappa $. 
Thus, 
by replacing $h$ with $h-\bar{\psi }(h)$, 
we may assume that $\bar{\psi }(h)=0$. 
Then, 
$h$ belongs to $\ker \bar{\psi }=q\kappa [\x ]$. 
Since $h$ is a coordinate of $\kappa [\x ]$, 
and hence irreducible, 
it follows that $h=cq$ 
for some $c\in \kappa ^*$. 
Therefore, 
we have $\ker \bar{\psi }=h\kappa [\x ]$.

We show that $h$ is $\G $-homogeneous. 
Write $h=\sum _{\gamma \in \Gamma _{\G }}h_{\gamma }$, 
where $h_{\gamma }\in \kappa [\x ]_{\gamma }$. 
Then, 
for each $\gamma \in \Gamma _{\G }$, 
we know by Lemma~\ref{lem:kspan} that 
$h_{\gamma }$ 
is a $k$-linear combination of 
$\delta _{\ba }(h)$ for $\ba \in \G $. 
For each $\ba \in \G $, 
we have 
$\phi _{\ba }:=\psi \circ \delta _{\ba }\circ \psi ^{-1}\in G$. 
Since $\phi _{\ba }$ is an element of $\Aut _R\Rx $ 
satisfying 
$\bar{\psi }\circ \bar{\delta }_{\ba }
=\bar{\phi }_{\ba }\circ \bar{\psi }$, 
we get 
$\bar{\psi }(\bar{\delta }_{\ba }(h))
=\bar{\phi }_{\ba }(\bar{\psi }(h))
=\bar{\phi }_{\ba }(0)=0$. 
Hence, 
$\delta _{\ba }(h)$ belongs to $\ker \bar{\psi }$ 
for each $\ba \in \G $. 
Thus, 
$h_{\gamma }$ belongs to $\ker \bar{\psi }=h\kapx $ 
for each $\gamma \in \Gamma _{\G }$. 
On the other hand, 
since no common monomials appear in 
$h_{\gamma }$ and $h_{\gamma '}$ if $\gamma \neq \gamma '$, 
we have $\deg h_{\gamma }\leq \deg h$ 
for each $\gamma \in \Gamma _{\G }$. 
Hence, 
$h=h_{\gamma }$ holds for some 
$\gamma \in \Gamma _{\G }$. 
Therefore, 
$h$ is $\G $-homogeneous. 
By Proposition~\ref{thm:homo autom}, 
there exist $\sigma _1=(h_1,h_2)\in C_{\G }(R)$, 
$a\in \kappa ^*$ and $i\in \{ 1,2\} $ 
such that $\det J\sigma _1=1$ and $a\bar{h}_i=h$. 
Then, 
$\psi (h_1)$ and $\psi (h_2)$ belong to $\Rx $. 
Moreover, 
$\psi (h_i)$ belongs to $p\Rx $, 
since $(\psi (h_i))^-
=\bar{\psi }(\bar{h}_i)=\bar{\psi }(a^{-1}h)=0$. 
Define $\sigma _2\in C_{\G }(K)$ 
by $\sigma _2(x_i)=p^{-1}h_i$ 
and $\sigma _2(x_j)=h_j$ for $j\neq i$. 
Then, 
$(\psi \circ \sigma _2)(x_i)=p^{-1}\psi (h_i)$ 
and 
$(\psi \circ \sigma _2)(x_j)=\psi (h_j)$ 
both belong to $\Rx $, 
and 
$$
\det J(\psi \circ \sigma _2)=\det J\psi \cdot \det J\sigma _2
=p^{-1}\det J\psi \cdot \det J\sigma _1 
=\alpha p_1\cdots p_{m-1}. 
$$
This contradicts the minimality of $m$, 
completing 
the proof of the case (1) of Theorem~\ref{thm:PID}~(i).

\medskip

Remark: 
The outline of the proof above 
is similar to the proof of \cite[Thm.\ 3.2]{KR}, 
but more precise treatments, 
such as Proposition~\ref{thm:homo autom}, 
are necessary when $k$ is not algebraically closed. 
In addition, 
the proof of \cite[Thm.\ 3.2]{KR} 
uses Sathaye~\cite[Thm.\ 1]{Sathaye} 
in a crucial step, 
which requires that $\ch k=0$.

\medskip

Finally, 
we give another proof of 
Lemma~\ref{lem:rat coord}
in the special case where $R$ contains $\Q $. 
Consider the $K$-derivation 
$D:\Kx \ni q\mapsto \det J(f,q)\in \Kx $. 
Let $g\in \Kx $ be such that $K[f,g]=\Kx $. 
Then, 
we have $D(f)=0$ 
and $D(g)=\det J(f,g)\in K^*$. 
Hence, 
$D$ is {\it locally nilpotent}, 
i.e., 
for each $q\in \Kx $, 
there exists $m\geq 1$ such that $D^m(q)=0$. 
Take $a\in K^*$ such that 
$aD(x_i)$ belongs to $\Rx $ for $i=1,2$. 
Then, 
$D_1:=aD$ restricts to an 
$R$-derivation of $\Rx $. 
We may choose $a$ so that $aD(x_1)$ or $aD(x_2)$ 
does not belong to $p\Rx $. 
Then, 
the 
$\kappa $-derivation 
$\bar{D}:=\id _{\kappa }\otimes _RD_1$ 
of $\kappa [\x ]$ is nonzero and locally nilpotent. 
Since $D_1(f)=0$, 
we have $\bar{D}(\bar{f})=0$. 
Hence, 
$\bar{f}$ belongs to $\ker \bar{D}$. 
On the other hand, 
there exists a coordinate $h$ of $\kappa [\x ]$ 
such that $\ker \bar{D}=\kappa [h]$ 
by Rentschler~\cite{Rentschler} 
(see also~\cite[Thm.\ 1.3.48]{Essen}), 
since the field $\kappa $ contains $\Q $. 
This proves Lemma~\ref{lem:rat coord}.

\section{Residual variables}\label{sect:Gauss}
\setcounter{equation}{0}

Throughout, 
let $R$ be a $k$-domain unless otherwise stated, 
and $G$ a subgroup of $\Aut _R\Rx $ 
such that $G_{(0)}$ is diagonalizable. 
Take 
$\psi =(f_1,\ldots ,f_n)\in \Aut _K\Kx $ 
and a subgroup $\G $ of $(k^*)^n$ with 
$\psi ^{-1}\circ G\circ \psi 
=\{ \delta _{\ba }\mid \ba \in \G \} $. 
We define an algebraic $\G $-action 
$V=(V_{\gamma })_{\gamma \in \Gamma _{\G }}$ 
on $\Rx $ as in (\ref{eq:intersection}).

Let $\p $ be a prime ideal of $R $. 
We say that $G$ {\it degenerates} at $\p $ 
if there exists $\gamma \in \Gamma _{\G }$ 
such that $V_{\gamma }\neq \zs $ 
and $\kappa (\p )\otimes _RV_{\gamma }=\zs $. 
If this is the case, 
$G$ degenerates at any prime ideal of $R$ 
containing $\p $.

The goal of this section 
is to prove the following theorem.

\begin{thm}\label{thm:non-degen}
Assume that $n\geq 2$, 
$R$ is a noetherian UFD over $k$, 
and $G$ does not degenerate at 
any maximal ideal of $R$. 
Then, 
Problem~$\ref{prob:main2}$ 
has an affirmative answer 
if at least $n-1$ of 
$\Gamma _{\G }^{(1)},\ldots ,\Gamma _{\G }^{(n)}$ 
are not equal to $\Gamma _{\G }$. 
\end{thm}

By the following proposition 
and the remark after Problem~\ref{prob:main}, 
Theorem~\ref{thm:non-degen} implies 
the case (2) of (i), 
and (ii) of Theorem~\ref{thm:main}.

\begin{prop}\label{prop:non-degenerate}
If $R$ is a regular $k$-domain, 
then $G$ does not degenerate 
at any prime ideal of $R$. 
\end{prop}
\begin{proof}
Take any prime ideal $\p $ of $R$ 
and $\gamma \in \Gamma _{\G }$. 
We show that 
$V_{\gamma }\neq \zs $ implies 
$\kappa (\p )\otimes _RV_{\gamma }\neq \zs $. 
Since $R$ is a domain, 
$R_{\p }\otimes _RV_{\gamma }$ 
contains $V_{\gamma }$. 
Hence, 
$V_{\gamma }\neq \zs $ implies 
$R_{\p }\otimes _RV_{\gamma }\neq \zs $. 
Thus, 
by replacing $R$ 
and $\p $ with $R_{\p }$ and $\p R_{\p }$, 
respectively, 
we may assume that $R$ is a regular local ring 
with maximal ideal $\p $. 
We show that 
$(R/\p )\otimes _RV_{\gamma }\neq \zs $ 
by induction on $r:=\dim R$. 
The assertion is clear if $r=0$. 
Assume that $r\geq 1$, 
and let $a_1,\ldots ,a_r\in \p $ 
be a regular system of parameters of $R$. 
Take any $f\in V_{\gamma }\sm \zs $. 
Since a regular local ring is a UFD 
(cf.~\cite[Thm.\ 20.3]{Matsumura}), 
there exists $i\geq 0$ 
such that $f_1:=a_1^{-i}f$ 
belongs to $\Rx \sm a_1\Rx $. 
Set $R_1:=R/a_1R$. 
Then, 
the image of $f_1$ in $R_1[\x ]$ 
is nonzero, 
and belongs to $R_1\otimes _RV_{\gamma }$. 
Note that $R_1$ 
is an $(r-1)$-dimensional 
regular local ring, 
and the maximal ideal $\p _1$ of $R_1$ 
is generated by the images of $a_2,\ldots ,a_r$ 
(cf.~\cite[Thm.\ 14.2]{Matsumura}). 
Hence, 
by induction assumption, 
we get 
$(R_1/\p _1)\otimes _RV_{\gamma }\neq \zs $. 
Since $R_1/\p _1\simeq R/\p $, 
it follows that $(R/\p )\otimes _RV_{\gamma }\neq \zs $. 
\end{proof}

The rest of this section 
is devoted to the proof of 
Theorem~\ref{thm:non-degen}. 
First, 
we investigate the structure of 
the algebraic $\G $-action 
$(\Kx _{\gamma })_{\gamma \in \Gamma _{\G }}$ 
on $\Kx $. 
For $i=1,\ldots ,n$, 
let $t_i$ be the minimal integer $t\geq 1$ 
with $t\gamma _i\in \Gamma _{\G }^{(i)}$, 
where $t_i:=\infty $ if 
$\Z \gamma _i\cap \Gamma _{\G }^{(i)}=\zs $. 
Note that $t_i=1$ 
if and only if $\Gamma _{\G }^{(i)}=\Gamma _{\G }$. 
Without loss of generality, 
we may assume that 
$t_i=\infty $ if $1\leq i\leq r$, 
$2\leq t_i<\infty $ if $r<i\leq s$ 
and $t_i=1$ if $s<i\leq n$ 
for some $0\leq r\leq s\leq n$. 
Set 
$\Lambda :=\sum _{i=r+1}^n\Z \gamma _i$. 
Then, 
we have 
$\Gamma _{\G }=\Lambda \oplus \bigoplus _{i=1}^r\Z \gamma _i$. 
Moreover, 
$t_i\gamma _i$ belongs to 
$\bigcap _{j=r+1}^n\Lambda _j$ 
for each $r<i\leq n$, 
where 
$\Lambda _j$ 
is the subgroup of $\Lambda $ 
generated by $\gamma _l$ for $r<l\leq n$ 
with $l\neq j$. 
For each $1\leq i\leq n$, 
we set $T_i:=\Z $ if $t_i=\infty $, 
and $T_i:=\{ 0,\ldots ,t_i-1\} $ if $t_i\neq \infty $. 
Then, 
we have the following lemma.

\begin{lem}\label{lem:standard grading}
Each $\gamma \in \Gamma _{\G }$ 
is uniquely written as 
\begin{equation}\label{eq:unique expression}
\gamma =\sum _{l=1}^si_l\gamma _l+\lambda, 
\ \text{ where }\ i_l\in T_l\ 
\text{ for }\ l=1,\ldots ,s\ 
\text{ and }\ 
\lambda \in \bigcap _{i=r+1}^n\Lambda _i. 
\end{equation}
We have 
$\Kx _{\gamma }=\Kx _{\lambda }
x_1^{i_1}\cdots x_s^{i_s}$ 
if $i_1,\ldots ,i_r\geq 0$, 
and $\Kx _{\gamma }=\zs $ 
otherwise. 
\end{lem}
\begin{proof}
There exist $j_1,\ldots ,j_n\in \Z $ 
such that $\sum _{l=1}^nj_l\gamma _l=\gamma $. 
For each $r<l\leq n$, 
let $q_l$ and $r_l$ be the quotient 
and remainder of $j_l$ divided by $t_l$, 
respectively. 
Then, 
$q_lt_l\gamma _l$ 
belongs to $\bigcap _{i=r+1}^n\Lambda _i$. 
Since $r_l=0$ if $s<l\leq n$, 
we obtain (\ref{eq:unique expression}) 
by setting 
$i_l:=j_l$ for $1\leq l\leq r$, 
$i_l:=r_l$ for $r<l\leq s$, 
and $\lambda :=\sum _{l=r+1}^nq_lt_l\gamma _l$. 
If $\gamma =\sum _{l=1}^si_l'\gamma _l+\lambda '$ 
is another expression, 
then we have 
$$(i_u-i_u')\gamma _u
=\sum _{l\neq u}(i_l'-i_l)\gamma _l+\lambda '-\lambda 
\in \Gamma _{\G }^{(u)}
$$ 
for each $1\leq u\leq s$. 
Since $i_u$ and $i_u'$ belong to $T_u$, 
it follows that $i_u=i_u'$ 
by the definition of $t_u$. 
This proves the uniqueness.

Clearly, 
$\Kx _{\gamma }$ contains 
$\Kx _{\lambda }
x_1^{i_1}\cdots x_s^{i_s}$ 
if $i_1,\ldots ,i_r\geq 0$. 
Hence, 
it suffices to check that 
$\Kx _{\gamma }\neq \zs $ implies 
$i_1,\ldots ,i_r\geq 0$ 
and $\Kx _{\gamma }\subset 
\Kx _{\lambda }
x_1^{i_1}\cdots x_s^{i_s}$. 
Assume that 
$x_1^{j_1}\cdots x_n^{j_n}$ belongs to $\Kx _{\gamma }$ 
for some $j_1,\ldots ,j_n\geq 0$. 
Then, 
we have $\sum _{l=1}^nj_l\gamma _l=\gamma $. 
This implies that $i_l=j_l$ for $1\leq l\leq r$ 
by the discussion above. 
Since $j_l\geq 0$, 
we get $i_l\geq 0$. 
Similarly, 
the quotient $q_l$ of $j_l$ 
divided by $t_l$ is nonnegative for $r<l\leq n$. 
Hence, 
$m:=\prod _{l=r+1}^nx_l^{q_lt_l}$ 
belongs to $\Kx _{\lambda }$. 
Therefore, 
$x_1^{j_1}\cdots x_n^{j_n}
=mx_1^{i_1}\cdots x_s^{i_s}$ 
belongs to $\Kx _{\lambda }
x_1^{i_1}\cdots x_s^{i_s}$. 
\end{proof}

We say that $f\in \Rx \sm \zs $ is {\it primitive} 
if no prime element $p$ of $R$ satisfies $f\in p\Rx $. 
We remark that, 
if $R$ is a UFD 
and $f\in \Rx \sm \zs $ is primitive, 
then $Bf\cap \Rx =(B\cap \Rx )f$ holds 
for any $K$-submodule $B$ of $\Kx $. 
In the situation of Theorem~\ref{thm:non-degen}, 
we may assume that $f_1,\ldots ,f_n$ 
are primitive elements of $\Rx $. 
Write $\gamma \in \Gamma _{\G }$ 
as in (\ref{eq:unique expression}), 
and assume that 
$i_1,\ldots ,i_r\geq 0$. 
Then, 
since $f_1^{i_1}\cdots f_s^{i_s}$ is primitive, 
we know by Lemma~\ref{lem:standard grading} 
that 
\begin{equation}\label{eq:Gauss}
\begin{aligned}
&V_{\gamma }=\psi (\Kx _{\gamma })\cap \Rx 
=\psi (\Kx _{\lambda }x_1^{i_1}\cdots x_s^{i_s})\cap \Rx \\
&\quad =\psi (\Kx _{\lambda })f_1^{i_1}\cdots f_s^{i_s}\cap \Rx 
=(\psi (\Kx _{\lambda })\cap \Rx )f_1^{i_1}\cdots f_s^{i_s}\\
&\quad 
=V_{\lambda }f_1^{i_1}\cdots f_s^{i_s}. 
\end{aligned}
\end{equation}

Now, 
take any $R$-algebra $S$, 
and let $\bar{f}_1,\ldots ,\bar{f}_n$ 
be the images of $f_1,\ldots ,f_n$ in $\Sx $. 
We consider the algebraic $\G $-action 
$V_S=(S\otimes _RV_{\gamma })_{\gamma \in \Gamma _{\G }}$ 
on $\Sx $. 
From (\ref{eq:Gauss}), 
it follows that 
$S\otimes _RV_{\gamma }=
(S\otimes _RV_{\lambda })
\bar{f}_1^{i_1}\cdots \bar{f}_s^{i_s}$ 
for each $\gamma \in \Gamma _{\G }$ 
with $i_1,\ldots ,i_r\geq 0$.

In the notation above, 
we have the following proposition.

\begin{prop}\label{prop:specialization}
Assume that $V_S$ is diagonalizable, 
and $S\otimes _RV_{\gamma }\neq \zs $ holds 
for each $\gamma \in \Gamma _{\G }$ with 
$V_{\gamma }\neq \zs $. 
Then, 
$\bar{f}_1,\ldots ,\bar{f}_s$ 
form a partial system of coordinates of $\Sx $. 
\end{prop}
\begin{proof}
By assumption, 
there exists $\sigma \in \Aut _S\Sx $ 
such that $y_i:=\sigma (x_i)$ 
is $V_S$-homogeneous for $i=1,\ldots ,n$. 
We show that, 
for each $1\leq l\leq s$, 
there exist $1\leq \sigma (l)\leq n$ 
and $\alpha _l\in \Sx ^*$ 
satisfying $\bar{f}_l=\alpha _ly_{\sigma (l)}$. 
First, 
we claim that there exists 
$1\leq \sigma (l)\leq n$ for which 
$y_{\sigma (l)}$ is written as 
$\bar{f}_1^{i_1}\cdots \bar{f}_s^{i_s}g$ 
with $i_l\geq 1$, 
where $i_u\in T_u$ for each $1\leq u\leq s$, 
and $g\in S\otimes _RV_{\lambda }$ for some 
$\lambda \in \bigcap _{i=r+1}^n\Lambda _i$. 
In fact, 
if not, 
$\Sx =S[y_1,\ldots ,y_n]$ 
is contained in 
$\bigoplus _{\gamma \in \Gamma _{\G }^{(l)}}
S\otimes _RV_{\gamma }$. 
Since $\gamma _l$ does not belong to 
$\Gamma _{\G }^{(l)}$ if $1\leq l\leq s$, 
we have $S\otimes _RV_{\gamma _l}=\zs $. 
This implies that 
$V_{\gamma _l}=\zs $ by assumption, 
contradicting $f_{\gamma _l}\in V_{\gamma _l}$. 
It remains only to check that 
$i_l=1$, $i_t=0$ for each $t\neq l$, 
and $g$ is a unit of $\Sx $. 
Take any prime ideal $\p $ of $S$, 
and let $\pi :\Sx \to (S/\p )[\x ]$ 
be the natural surjection. 
Then, 
$\pi (y_{\sigma (l)})
=\pi (\bar{f}_1)^{i_1}\cdots 
\pi (\bar{f}_s)^{i_s}
\pi (g)$ 
is a coordinate of $(S/\p )[\x ]$, 
and hence is an irreducible element of $(S/\p )[\x ]$. 
Now, 
consider the algebraic $\G $-action 
$V_{S/\p }$ on $(S/\p )[\x ]$. 
Then, 
$\pi (\bar{f}_i)$ belongs to 
$(S/\p )\otimes _RV_{\gamma _i}$ 
for each $i$. 
If $1\leq i\leq s$, 
then 
$\bigl((S/\p )\otimes _RV_{\gamma _i}\bigr) 
\cap (S/\p )$ equals $\zs $, 
since $\gamma _i\neq 0$. 
Hence, 
we have either $\pi (\bar{f}_i)=0$ 
or $\pi (\bar{f}_i)\not\in S/\p $. 
By the irreducibility of $\pi (y_{\sigma (l)})$, 
it follows that 
$i_l=1$, $i_t=0$ for each $t\neq l$, 
and $\pi (g)$ belongs to $(S/\p )[\x ]^*=(S/\p )^*$. 
Since $\p $ is any prime ideal of $S$, 
we know that 
the constant term $c$ of $g$ 
is a unit of $S$, 
and $g-c$ is a nilpotent element of $\Sx $. 
Therefore, 
$g$ is a unit of $\Sx $. 
\end{proof}

Let us complete the proof of Theorem~\ref{thm:non-degen}. 
We may assume that 
$\Gamma _{\G }^{(i)}\neq \Gamma _{\G }$ 
for each $i\neq n$. 
Then, 
we have $t_i\neq 1$ for each $i\neq n$. 
Thanks to Lemma~\ref{lem:partial system}, 
it suffices to show that 
$f_1,\ldots ,f_{n-1}$ 
form a partial system of coordinates of $\Rx $. 
Take any prime ideal $\p $ of $R$. 
By assumption, 
$G$ does not degenerate at maximal ideals 
of $R$ containing $\p $. 
Hence, 
$G$ does not degenerate at $\p $. 
Moreover, 
$G_{\p }$ is diagonalizable 
by the assumption of Problem~\ref{prob:main2}. 
Hence, 
by Proposition~\ref{prop:specialization}, 
the images of $f_1,\ldots ,f_{n-1}$ 
in $\kappa (\p )[\x ]$ 
form a partial system of coordinates of $\kappa (\p )[\x ]$. 
This implies that 
$f_1,\ldots ,f_{n-1}$ 
form a partial system of coordinates of $\Rx $ 
thanks to the result on ``residual variables" 
by Bhatwadekar-Dutta~\cite[Remark 3.4]{BD}, 
since $R$ is a noetherian UFD by assumption, 
and UFD is seminormal.

\noindent
Department of Mathematics and Information Sciences\\ 
Tokyo Metropolitan University \\ 
1-1  Minami-Osawa, Hachioji, 
Tokyo 192-0397, Japan\\
kuroda@tmu.ac.jp

\end{document}